\documentclass{amsart}
\usepackage{amssymb,amsfonts, color,epsf}
\usepackage{mathabx}
\usepackage [cmtip,arrow]{xy}
\xyoption{all}
\usepackage {pb-diagram,pb-xy}
\usepackage{enumerate}
\usepackage{accents}
\usepackage[noautoscale]{youngtab}
\usepackage{subfigure}

\ifx\pdfpageheight\undefined
   \usepackage[dvips,colorlinks=true,linkcolor=blue,citecolor=red,%
      urlcolor=green]{hyperref}
   \usepackage[dvips]{graphicx}
   \makeatletter
   \edef\Gin@extensions{\Gin@extensions,.mps}
   \makeatother
\else
 \usepackage[pdftex]{graphicx}
   \usepackage[bookmarksopen=false,pdftex=true,breaklinks=true,%
      backref=page,pagebackref=true,plainpages=false,%
      hyperindex=true,pdfstartview=FitH,colorlinks=true,%
      pdfpagelabels=true,colorlinks=true,linkcolor=blue,%
      citecolor=red,urlcolor=green,hypertexnames=false%
      ]%
   {hyperref}
\fi
\usepackage{bm}

\usepackage{tikz-cd}
\makeatletter
\tikzset{
  column sep/.code=\def\pgfmatrixcolumnsep{\pgf@matrix@xscale*(#1)},
  row sep/.code   =\def\pgfmatrixrowsep{\pgf@matrix@yscale*(#1)},
  matrix xscale/.code=%
    \pgfmathsetmacro\pgf@matrix@xscale{\pgf@matrix@xscale*(#1)},
  matrix yscale/.code=%
    \pgfmathsetmacro\pgf@matrix@yscale{\pgf@matrix@yscale*(#1)},
  matrix scale/.style={/tikz/matrix xscale={#1},/tikz/matrix yscale={#1}}}
\def\pgf@matrix@xscale{1}
\def\pgf@matrix@yscale{1}
\makeatother

\newcommand {\hide}[1]{}

\usepackage{amsmath,amssymb,amsfonts}
\newtheorem{theorem}{Theorem}
\newtheorem{lemma}{Lemma}[section]
\newtheorem{corollary}{Corollary}
\newtheorem{proposition}{Proposition}[section]
\newtheorem{claim}{Claim}[section]
\newtheorem*{claim*}{Claim}
\newtheorem{conjecture}{Conjecture}

\newtheorem*{theorem*}{Theorem}
\newtheorem*{corollary*}{Corollary}
\theoremstyle{definition}
\newtheorem{definition}{Definition}[section]
\newtheorem{example}{Example}[section]

\newtheorem{notation}{Notation}[section]

\newtheorem{remark}{Remark}[section]

\usepackage{algorithm}
\usepackage{algpseudocode}

\usepackage{tikz}

\title[Topology of real multi-affine hypersurfaces ]{Topology of real multi-affine
  hypersurfaces and a homological stability property}
  
\author{Saugata Basu}
\address{Department of Mathematics,
Purdue University, West Lafayette, IN 47906, U.S.A.}
\email{sbasu@math.purdue.edu}

\author{Daniel Perrucci}
\address{Departamento de Matemática
Facultad de Ciencias Exactas y Naturales
Universidad de Buenos Aires, 
Buenos Aires, Argentina.
}
\email{perrucci@dm.uba.ar}

\def\R{\mathrm{R}}
\def\N{\mathbb{N}}

\def\Q{\mathbb{Q}}
\def\Zz{\mathbb{Z}}
\def\Z{{\rm{Z}}}
\def\H{\mathrm{H}}

\def\H{\mathbf{H}}

\def\Sphere{\mathbf{S}}
\def\RR{\mathcal{R}}
\def\card{\mathrm{card}}
\def\Hom{\mathrm{Hom}}
\def\length{\mathrm{length}}
\def\Ind{\mathrm{Ind}}
\def\mult{\mathrm{mult}}

\begin{document}
\begin{abstract}
Let $\R$ be a real closed field.
We prove that the number of semi-algebraically connected components of
a real hypersurface in $\R^n$  defined by a multi-affine polynomial 
of degree $d$ is bounded by $2^{d-1}$. This bound is sharp and is independent of $n$ (as opposed to the classical bound of $d(2d -1)^{n-1}$ 
on the Betti numbers
of hypersurfaces defined by arbitrary polynomials of degree $d$ in 
$\R^n$ due to Petrovski{\u\i} and Ole{\u\i}nik, Thom and Milnor). Moreover, we show 
there exists $c > 1$, 
such that given a sequence $(B_n)_{n >0}$ where $B_n$ is a closed ball in $\R^n$ of positive radious,  there exist
hypersurfaces $(V_n)_{n_>0}$ defined by symmetric multi-affine polynomials of degree $4$, such that $\sum_{i \leq 5} b_i(V_n \cap B_n) > c^n$,
where $b_i(\cdot)$ denotes the $i$-th Betti number with rational coeffcients.
Finally, 
as an application of the main result of the paper
we verify a representational stability conjecture due to Basu and Riener on the 
cohomology modules of symmetric real algebraic sets for a new and much larger class of 
symmetric real algebraic sets than known before.
\end{abstract}

\maketitle
\section{Introduction}
We fix a real closed field $\R$.
For any 
closed semi-algebraic set $S \subset \R^n$, we denote by $b_i(S)$ the dimension of the $i$-th homology group $\H_i(S)$ with rational coefficients
(the $i$-th Betti number of $S$)\footnote{Since we only consider homology and cohomology groups of semi-algebraic sets with rational coeffcients we have $\H_i(S) \cong \H^i(S)$ for any closed semi-algebraic set $S$ by the universal coefficients theorem \cite[page 243]{Spanier}.}. We will denote 
\[
b(S) = \sum_{i \geq 0} b_i(S).
\]
In particular, $b_0(S)$ equals the number of semi-algebraically connected components of $S$.

The problem of proving upper bounds on the Betti numbers of a real algebraic
variety $V \subset \R^n$
in terms of the degrees of polynomials defining $V$ 
is a very well-studied problem in real algebraic geometry. 
Before stating the classical result in this direction 
it is useful to first introduce some notation that we are also going to use later in the paper.

\begin{notation}
For $\mathcal{P}$ a finite subset of $\R[X_1, \dots, X_n]$, 
$B$ a semi-algebraic subset of $\R^n$,
we denote by $\Z(\mathcal{P}, B)$ the set of common zeros of 
$\mathcal{P}$ in $B$. If $\mathcal{P} = \{P \}$, will denote
$\Z(\mathcal{P}, B)$ by $\Z(P, B)$.
If $\phi$ is a quantifier-free formula in the first order theory of the reals (i.e. a Boolean combination of atoms of the form $P \geq 0, P \in \R[X_1,\ldots,X_n]$), then we will denote by $\RR(\phi,B)$ the semi-algebraic subset of $B$ defined by $\phi$.
\end{notation}

An upper bound on the sum of the Betti numbers of a real algebraic set
in $\R^n$ in terms of the degrees of its defining polynomials was proved
by Petrovski{\u\i} and Ole{\u\i}nik \cite{OP}, Thom \cite{T}, and Milnor \cite{Milnor2}. However, the proof of this result actually proves an a priori stronger result, namely a bound on the Betti numbers of the intersection of
the real algebraic set with any closed Euclidean ball in $\R^n$ (see for example proof of Proposition 7.28 in \cite{BPRbook2}). In order to make explicit this distinction we introduce the following notation.

\begin{definition}
\label{def:G}
Let $\mathbf{B} = (B_n)_{n > 0}$ be a sequence of  
closed semi-algebraic subsets of $\R^n$, and
$\mathbf{F} = (\mathcal{F}_n)_{n > 0}$, a sequence where
$\mathcal{F}_n \subset \R[X_1,\ldots,X_n]$ for each $n >0$.
We define for each $p \geq 0$,
\begin{eqnarray*}
\beta_{\mathbf{F},\mathbf{B},p}(n) &=& \max_{P \in \mathcal{F}_n} \left(\sum_{i \leq p}b_i(\Z(P,B_n))\right), \\
\beta_{\mathbf{F},p}(n) &=& 
\beta_{\mathbf{F},(\R^n)_{n >0},p}(n),
\end{eqnarray*}
and also define,
\begin{eqnarray*}
\beta_{\mathbf{F},\mathbf{B}}(n) &=& \beta_{\mathbf{F},\mathbf{B},n}(n),\\ 
\beta_{\mathbf{F}}(n) &=& \beta_{\mathbf{F},n}(n).
\end{eqnarray*}
\end{definition}

We next observe that under certain  conditions on
$\mathbf{B}$ and $\mathbf{F}$, 
$\beta_{\mathbf{F},\mathbf{B},p}(n)$ (respectively, $\beta_{\mathbf{F},\mathbf{B}}(n)$) is an upper bound on
$\beta_{\mathbf{F},p}(n)$ (respectively, $\beta_{\mathbf{F}}(n)$).

Following the notation in Definition~\ref{def:G}: 
\begin{proposition}
\label{prop:bounded-vs-unbounded}
Suppose that for each $n > 0$, 
$B_n$ is a closed and bounded convex semi-algebraic set having dimension $n$, and
$\mathcal{F}_n$ is closed under translations
$X \mapsto  X - x, x \in \R^n$, and scalings $X \mapsto \lambda \cdot X, \lambda \in \R$.
Then, 
\begin{eqnarray*}
\beta_{\mathbf{F},p}(n) &\leq & \beta_{\mathbf{F},\mathbf{B},p}(n), \mbox{ for $p \geq 0$, and}\\
\beta_{\mathbf{F}}(n) &\leq & \beta_{\mathbf{F},\mathbf{B}}(n). \\
\end{eqnarray*}
\end{proposition}

\begin{proof}
Since  $\mathcal{F}_n$ is stable under translations 
and $\dim B_n = n$,
one can assume  that $B_n$ contains the origin in its interior.
As $B_n$ is convex, this implies that for $\lambda > 0$, 
$\lambda \cdot B_n$ is an increasing family of semi-algebraic sets (increasing with
$\lambda$),
$\R^n =  \bigcup_{\lambda > 0} \lambda\cdot B_n$, and each $\lambda \cdot B_n$ is a closed 
and bounded 
semi-algebraic set.
It follows from the 
conic structure theorem at infinity of 
semi-algebraic sets (see for instance \cite[Proposition 5.49]{BPRbook2})
that there exists
$\lambda_0 > 0$, such that $\Z(P, \lambda \cdot B_n)$ is a semi-algebraic
deformation retract of $\Z(P,\R^n)$. 

Now let $P_0 = P(\lambda_0\cdot X_1,\ldots, \lambda_0\cdot X_n) \in \mathcal{F}_n$.
Then, $\Z(P_0,B_n)$ is semi-algebraically homeomorphic to 
$\Z(P, \lambda_0\cdot B_n)$. Hence,
\begin{eqnarray*}
b_p(P_0,B_n) &=& b_p(\Z(P,\R^n)), \mbox{ for $p \geq 0$}, \\
b(P_0,B_n) &=& b(\Z(P,\R^n)).
\end{eqnarray*}
This proves both inequalities in the proposition.
\end{proof}

\begin{remark}
\label{rem:prop:bounded-vs-unbounded}
The inequalities in Proposition~\ref{prop:bounded-vs-unbounded} can
be strict. Take for example, 
\[
\mathbf{F} = \left(\R[X_1,\ldots,X_n]_{\leq 2}\right)_{n > 0}
\]
(where $\R[X_1,\ldots,X_n]_{\leq d}$ denotes the subset polynomials of degree at most $d$),
and 
\[
\mathbf{B} = ([-1,1]^n)_{n > 0}.
\]

Then for $n \geq 2$,
\begin{eqnarray*}
\beta_{\mathbf{F},0}(n) &=& 2, \\
\beta_{\mathbf{F},\mathbf{B},0}(n) &\geq& 2^n.
\end{eqnarray*}

The first equation is obvious. For the second inequality,
consider
\[
P_n = \sum_{i=1}^n X_i^2 - n.
\]
Then,
$\Z(P_n,B_n) = \{-1,1\}^n$, and thus
$b_0(\Z(P_n,B_n)) = 2^n$.
\end{remark}

The theorem of Petrovski{\u\i} and Ole{\u\i}nik \cite{OP}, Thom \cite{T}, and Milnor \cite{Milnor2} can now be restated as follows.

\begin{theorem}[Petrovski{\u\i} and Ole{\u\i}nik \cite{OP}, Thom \cite{T}, and Milnor \cite{Milnor2}]
\label{thm:OPTM}
For each $n >0$, let $B_n$ be a closed Euclidean ball in $\R^n$ of positive radius, and 
\[
\mathbf{F}_d = \left(\R[X_1,\ldots,X_n]_{\leq d}\right)_{n > 0}.
\]
Then,
\begin{eqnarray*}
\beta_{\mathbf{F}_d,\mathbf{B}}(n) &\leq& d(2d-1)^{n-1}.
\end{eqnarray*}
\end{theorem}

Using Proposition~\ref{prop:bounded-vs-unbounded} one immediately obtains from Theorem~\ref{thm:OPTM} the following corollary.
\begin{corollary}
\label{cor:OPTM}
\begin{equation}
\label{eqn:cor:OPTM}
\beta_{\mathbf{F}_d}(n) \leq  d(2d-1)^{n-1}.
\end{equation}
\end{corollary}

Note that the upper bound in \eqref{eqn:cor:OPTM} grows exponentially in
$n$ for $d$ fixed.
Another point to note is that the proofs of the upper bounds on the sum of the Betti numbers in \eqref{eqn:cor:OPTM}
ultimately rely on 
bounding the number of critical points of certain Morse functions. As such 
it does not give any additional information on a specific Betti number (say the 
zero-th Betti number). In fact the problem of proving bounds on individual
Betti numbers which are better than the bounds on the sum of all Betti numbers
is of great interest in real algebraic geometry. One of the main results
in this paper (Theorem~\ref{thm:ccez} below) furnishes such a bound (on the zero-th Betti number) for a special class of real algebraic hypersurfaces in $\R^n$
that we define below.

\subsection{Multi-affine polynomials}
We consider real algebraic varieties in $\R^n$ defined by 
polynomials of a special shape.

\begin{definition}
We call $P \in \R[X_1, \dots, X_n]$ a \emph{multi-affine} polynomial 
if for every $i, 1 \leq i \leq n$, $\deg_{X_i}P \le 1$. We denote the 
subset of multi-affine polnomials in $\R[X_1, \dots, X_n]_{\leq d}$
by $\mathcal{A}_{d,n}$, and the sequence 
$(\mathcal{A}_{d,n})_{n > 0}$ by $\mathbf{A}_d$.
\end{definition}

Real multi-affine polynomials occur in several applications. For example, 
multi-affine polynomials appear 
in  computational complexity theory, since every element of the coordinate ring, 
\[
\R[\mathcal{B}_n] = \R[X_1,\ldots,X_n]/(X_1(X_1-1), \ldots, X_n(X_n-1)),
\]
of the Boolean hypercube, $\mathcal{B}_n = \{0,1\}^n$, 
can be represented by a multi-affine polynomial. 
The smallest degree
of the unique multi-affine polynomial representing a Boolean function
$f:\mathcal{B}_n \rightarrow \{0,1\}$ is called the degree of $f$,
and is used as a measure of complexity of $f$ \cite{Nisan-Szegedy}.

The multi-affine polynomial 
\[
P_{\mathcal{M}} = \sum_{I} {X^I} \in \R[X_1,\ldots,X_n],
\]
where for $I \subset [1,n]$, $X^I$ denotes the monomial
$\prod_{i \in I} X_i$, and 
where $I$ varies over the bases of a matroid $\mathcal{M}$,
is called the \emph{basis generating polynomial of $\mathcal{M}$}.
Its properties (such as real stability) play 
an important role in the study of matroids,
for instance in the spectacular recent works by 
Anari et al. \cite{Anari-et-al} and Br\"{a}nd\'{e}n and Huh \cite{Branden-Huh}. However,
the topology of the real hypersurfaces they define has not been studied much to the best of our knowledge. In this paper we prove quantitative results on certain  topological
invariants (Betti numbers) of real hypersurfaces defined by multi-affine polynomials of any fixed degree.

Finally, elementary symmetric polynomials, as well as linear combinations of them, furnish
examples of multi-affine polynomials. This last class of polynomials,
which are multi-affine as well as symmetric, 
will appear again later in the paper.

\section{Main Results}
We now state the main results proved in this paper in the following
three subsections.

\subsection{Bound on the zero-th Betti number}
\label{subsec:zero}
Our first result is a bound on the number of semi-algebraically connected components of hypersurfaces in $\R^n$ defined by multi-affine polynomials which is independent 
of $n$. More precisely, we prove the following theorem.

\begin{theorem}
\label{thm:ccez}
$$
\beta_{\mathbf{A}_d,0}(n) \leq 2^{d-1}.
$$
\end{theorem}

\begin{example}[Sharpness]
The bound in Theorem~\ref{thm:ccez} is sharp: for $d, n \in \N$, 
let 
$P = X_1\dots X_d - 1 \in 
\mathcal{A}_{d,n}
$.
Then, 
$b_0(\Z(P, \R^n)) = 2^{d-1}$.
\end{example}

\begin{remark}
Unlike the proof of Corollary~\ref{cor:OPTM} above, we will prove Theorem~\ref{thm:ccez} directly without first proving a bounded version. 
\end{remark}

\subsection{Varieties defined by more than one polynomials and 
higher Betti numbers}

A remarkable property of the bound in Theorem~\ref{thm:ccez} is that it is independent of $n$ (unlike the bound in \eqref{eqn:cor:OPTM}). 
However, there are two restrictive features of the bound in Theorem~\ref{thm:ccez} that are worth pointing out.

\begin{enumerate} [(a)]
\item
    \label{itemlabel:features:a}
    The bound applies only to varieties defined by a single multi-affine polynomial.
    Note that the usual trick in real algebraic geometry of reducing the number of polynomials defining a variety to one by taking a sum of squares does not work well with the class of multi-affine polynomials. The square of a multi-affine polynomial is no longer necessarily multi-affine.
    \item 
    \label{itemlabel:features:b}
    The bound in 
Theorem~\ref{thm:ccez} applies only to the zero-th Betti number (as opposed to the
sum of all the Betti numbers).
    
\end{enumerate}

It is natural to ask whether one could improve Theorem~\ref{thm:ccez} by removing
the restrictions \eqref{itemlabel:features:a} and \eqref{itemlabel:features:b}.
We show that this is not possible if we want to have an upper bound that is independent of $n$ (in the case of restriction \eqref{itemlabel:features:b} our result only applies to the bounded version -- see Theorem~\ref{thm:example}). 

We first address \eqref{itemlabel:features:a}. 
We construct below a sequence of examples each involving three multi-affine polynomials in $\R[X_1,\ldots,X_n]$ of degree at most $4$, such that the number of connected components of the real variety they define grows with $n$. In order to construct these polynomials we need to introduce some notation.

\begin{notation}
\label{not:elementary}
  For $n \in \N_0$ and $\ell \in \Zz$ with $\ell \ge -1$, we denote 
  by 
 $\sigma_{\ell, n} \in \R[X_1, \dots, X_n]$ 
the $\ell$-th elementary symmetric polynomial in $X_1,\ldots,X_n$ defined as follows: 
 
 \begin{itemize}
  \item $\sigma_{-1, n} = 0$, 
  \item $\sigma_{0, n} = 1$, 
  \item $\sigma_{\ell, n} = \sum_{1 \le i_1 < \dots < i_\ell \le n} X_{i_1} \dots X_{i_\ell}$ for $1 \le \ell \le n$, 
 \item $\sigma_{\ell, n} = 0$ for $\ell > n$. 
 \end{itemize}
\end{notation}

It is clear that for $n \in \N$ and $\ell \in \Zz$ with $\ell \ge -1$,
$\sigma_{\ell,n}$ is multi-affine. Also, for $0 \le \ell \le n$, 
 $$
 \sigma_{\ell, n} = X_n\sigma_{\ell-1, n-1} + \sigma_{\ell, n-1}.
 $$

\begin{notation}
\label{not:Newton}
For $\ell, n \in \N$, we denote by $N_{\ell,n}$
the $\ell$-th power sum polynomial,
\[
N_{\ell,n} = X_1^\ell + \cdots + X_n^\ell.
\]
\end{notation}

When the value of $n$ is clear from the context, we will simply write
$\sigma_\ell$ to denote $\sigma_{\ell, n}$
and $N_\ell$ to denote $N_{\ell, n}$.

\begin{example}\label{ex:three_pols}
Consider any fixed value of $k \in \N$. 
For $n \ge k$, 
consider 
$P_1, P_2, P_3 \in \R[X_1, \dots, X_n]$ of degree bounded by $d = 4$:
$$
\begin{array}{ccl}
P_1(X) & = & \sigma_{1}(X) - k, \\[2mm]
P_2(X) & = & \sigma_{2}(X) - \frac12 k(k-1), \\[2mm]
P_3(X) & = & (4k-6)\sigma_{3}(X) -4\sigma_{4}(X)  - \frac12k(k-1)^2(k-2). \\[2mm]
\end{array}
$$

Using the Newton identities
$$
\begin{array}{ccl}
N_1 & = & \sigma_1, \\[2mm]
N_2 & = & N_1\sigma_1  -  2 \sigma_2, \\[2mm]
N_3 & = & N_2\sigma_1  - N_1\sigma_2  + 3\sigma_3, \\[2mm]
N_4 & = & N_3\sigma_1  - N_2\sigma_2 + N_1\sigma_3  - 
4\sigma_4,\\[2mm]
\end{array}
$$
for $x \in Z(\{P_1, P_2, P_3\}, \R^n)$ we have
$$
N_1(x) = k,
$$
$$
N_2(x) = k, 
$$
and
\begin{eqnarray*}
\sum_{1 \le i \le n} x_i^2(x_i - 1)^2 &=& N_4(x) - 2 N_3(x) + N_2(x) \\
&=& (4k-6)\sigma_3(x) -4\sigma_4(x) - \frac12k(k-1)^2(k-2) \\
&=& 0.
\end{eqnarray*}
This implies that $\Z(\{P_1, P_2, P_3\}, \R^n)$ is a finite set with $\binom{n}{k}$ points: each point
is an element of $\{0, 1\}^{n}$ with exactly $k$ coordinates equal to $1$ and $n-k$ coordinates equal to $0$. 
Therefore 
$$
b_0(\Z(\{P_1, P_2, P_3\}, \R^n)) = \binom{n}{k}
$$
cannot be bounded only in terms of $d = 4$ (independently from $n$). 
\end{example}

Example \ref{ex:three_pols} shows that it is impossible to obtain a bound 
on the number of semi-algebraically connected components of a real variety in
$\R^n$ 
defined by three multi-affine polynomials which is independent of $n$.
We do not know if such a bound exists for a
real variety defined by two multi-affine polynomials of degree at most $d$.

We now address \eqref{itemlabel:features:b}. 
We first introduce a notation. 

\begin{notation}
For $0 \leq d \leq n$ we denote
\[
\Sigma_{d,n} =  \left\{P \in \R[X_1,\ldots,X_n]_{\leq d} \;\mid\; P = \sum_{0 \le i \le d} a_i \sigma_{i,n} , a_i \in \R, 0 \leq i \leq d \right\}.
\]
Moreover, we denote $\mathbf{\Sigma}_d = (\Sigma_{d,n})_{n >0}$.
\end{notation}

It is natural to wonder whether one can
obtain a bound on $\beta_{\mathbf{A}_d}(n)$ that is independent of $n$.
We prove the following theorem which rules out a bound independent of $n$ for the intersection of these hypersurfaces with bounded closed balls.

\begin{theorem}
\label{thm:example}
There exists a constant $c > 1$ having the following property.
Let $\mathbf{B} = (B_n)_{n >0}$, where each $B_n$ is a symmetric, closed, convex, bounded semi-algebraic subset of $\R^n$ with $\dim B_n =n$.
Then for $n > 1$,
\[
\beta_{\mathbf{\Sigma}_{4},\mathbf{B},5}(n) > c^n.
\]
In particular, since for each $d, n > 0$, 
$\Sigma_{d,n} \subset \mathcal{A}_{d,n}$, we also have for $n > 1$,
\[
\beta_{\mathbf{A}_{4},\mathbf{B},5}(n) > c^n.
\]
\end{theorem}

\begin{remark}
\label{rem:Agrachev}
The proof of Theorem~\ref{thm:example} uses two different ingredients 
and will be given in Section~\ref{subsec:proof:thm:example}.
First, it uses representation theory of the symmetric 
group.
Second, it uses a certain spectral sequence argument originally
used by Agrachev \cite{Agrachev, Ag}, and later by other authors 
\cite{BP'R07jems, Agrachev-Lerario,Lerario2014} for proving \emph{upper bounds} 
on the Betti numbers
of semi-algebraic sets defined by quadratic inequalities (in the non-symmetric situation). We use it in this paper 
for proving  \emph{lower bounds} on the maximum Betti numbers 
occurring in a family symmetric real varieties 
(i.e. for proving existence of symmetric real varieties with large Betti numbers). This technique of proof 
might be of independent interest for proving lower bounds on the
maximum possible Betti number of real varieties defined by other families of (symmetric) polynomials than those we consider in this paper. 
\end{remark}

\begin{remark}
Also, note that proving existence of real varieties with maximum possible Betti numbers is a well studied problem in real algebraic geometry
(see \cite{Itenberg-Viro2007, Bertrand2006}). Theorem~\ref{thm:example} is distinguished from these results because of several reasons.
\begin{enumerate}[(a)]
    \item The results in the papers cited above are about real projective or more generally toric varieties, while we study real affine varieties
    in this paper.
    \item The asymptotics in the above cited papers are for fixed $n$, with the degree of the polynomial tending to infinity. In this paper, we consider the degree to be fixed and let $n$ be large.
    \item 
    Finally, it is not clear if the method of ``combinatorial patchworking'' used in \cite{Itenberg-Viro2007, Bertrand2006} can be used to construct real \emph{symmetric} varieties having large Betti numbers. 
\end{enumerate}
\end{remark}

\subsection{Stability conjecture}
\label{subsec:stability}
We now describe a connection between the results stated above with the study of
the cohomology groups of symmetric semi-algebraic sets as modules over the 
symmetric group. 

\subsubsection{Some background}
The symmetric group $\mathfrak{S}_n$ acts on $\R^n$ by permuting 
coordinates. We say that a semi-algebraic subset $S \subset \R^n$ is symmetric if it is
stable under this action.
The action of $\mathfrak{S}_n$
on a closed symmetric semi-algebraic set
$S \subset \R^n$ 
induces an an action on the cohomology $\H^*(S)$, 
giving $\H^*(S)$ the structure of 
\emph{a finite dimensional $\mathfrak{S}_n$-module}. 

\begin{remark}
\label{rem:hom-vs-cohom}
Note that if $\H$ is a finite-dimensional  $\mathfrak{S}_n$ module (over $\Q$), 
$\Hom(\H,\Q)$ has a canonically defined induced $\mathfrak{S}_n$-module
structure, and is isomorphic to $\H$ as an $\mathfrak{S}_n$-module.
\footnote{This is a consequence of the fact that the group $\mathfrak{S}_n $ is ambivalent; every element is conjugate to its inverse.}

Also, using the universal coefficient theorem, we have that for any closed
semi-algebraic set $S \subset \R^n$, 
$\H^i(S) \cong \Hom(\H_i(S),\Q)$. If $S$ is additionally symmetric, then we have that
$\H^i(S) \cong_{\mathfrak{S}_n} \H_i(S)$.
\end{remark}

General facts from group representation theory then tell us that
the $\mathfrak{S}_n$-module $\H^*(S)$ admits a canonically defined  \emph{isotypic decomposition} as a direct sum
of sub-$\mathfrak{S}_n$-submodules, each of which is a multiple of a certain irreducible $\mathfrak{S}_n$-module.
The irreducible $\mathfrak{S}_n$-modules are well studied, and they are in bijection with the finite set of partitions
of the number $n$ -- the module corresponding to the partition $\lambda \vdash n$ will be denoted by $\mathbb{S}^\lambda$ in what follows, and is called the \emph{Specht}-module corresponding to $\lambda$
(see the book \cite{James-Kerber} for the precise definitions of these objects).
We will use the following notation.

 \begin{notation}
 \label{not:isotypic}
 For any finite dimensional $\mathfrak{S}_n$-module $\H$,
 \footnote{The choice of $\H$ to denote the representation is deliberate since all the $\mathfrak{S}_n$-modules considered in this paper will be of the form $\H^*(V)$ or $\H_*(V)$ for some symmetric real algebraic set $V$. }
 and $\lambda = (\lambda_1, \dots, \lambda_\ell) \vdash n$, we will
 denote by $ \H_\lambda $ the isotypic component corresponding to the 
 Specht module $\mathbb{S}^\lambda$ in $\H$. Thus, the isotypic decomposition of $\H$ is the direct sum decomposition
 \[
 \H \cong_{\mathfrak{S}_n} \bigoplus_{\lambda \vdash n} \H_\lambda,
 \]
 and each $\H_\lambda \cong_{\mathfrak{S}_n} m_\lambda \mathbb{S}^\lambda$,
 where $m_\lambda \geq 0$. We will denote $\mult_\lambda(\H) = m_\lambda$.
 \end{notation}

Thus the isotypic decomposition of $\H^*(S)$ gives a canonically defined direct sum decomposition (direct sum in the category of  $\mathfrak{S}_n$-modules)
\begin{equation}
\label{eqn:def:multiplicity}
\H^i(S) \cong_{\mathfrak{S}_n} \bigoplus_{\lambda \vdash n} m_{i,\lambda}(S)  \mathbb{S}^\lambda,
\end{equation}
where 
\[
m_{i,\lambda}(S) = \mult_{\lambda}(\H^i(S)).
\]

The dimension
of the Specht module $\mathbb{S}^\lambda$, has a simple expression
(see for example \cite[Theorem 2.3.21]{James-Kerber}
\begin{equation}
\label{eqn:hook}
\dim  \mathbb{S}^\lambda =  
\frac{n!}{\prod_{i,j} h_{i,j}(\lambda)},
\end{equation}
where 
\[
h_{i,j}(\lambda) = \lambda_i + \lambda'_j -i - j + 1,
\]
and $\lambda'$ is the transpose of $\lambda$.
\footnote{
Since $h_{i,j}(\lambda)$ in the above formula is equal to the
length of the hook with corner in the box $(i,j)$ in the Young diagram of $\lambda$, the formula \eqref{eqn:hook}
is often  called the \emph{hook length formula}. 
}
Note that these dimensions can be exponentially big even for relatively simple partitions (say the partition $(n/2,n/2)$ for even $n$). 
For a symmetric semi-algebraic set $S \subset \R^n$,
knowing the multiplicities $m_{i,\lambda}(S), \lambda \vdash n$,  allows one to compute the dimension of $\H^i(S)$,
and thus the $i$-th Betti number of $S$ (using Eqn. \eqref{eqn:hook}).

The partition $(n) \vdash n$ having length one plays a special role. The corresponding Specht-module
$\mathbb{S}^{(n)}$ is the one dimensional \emph{trivial representation}  of $\mathfrak{S}_n$, and the isotypic
component of $\H^i(S)$ corresponding to  the partition $(n)$ is thus isomorphic to the fixed part
$\H^i(S)^{\mathfrak{S}_n}$ of $\H^i(S)$, which in turn is isomorphic to 
$\H^i(S/\mathfrak{S}_n)$ (see \cite{BC-imrn} for details and subtleties regarding these isomorphisms). We will use this last fact later in the paper (in the proof of Proposition~\ref{prop:stable_con_comp}).

The decomposition of the cohomology modules of a closed semi-algebraic set $S \subset \R^n$
defined by symmetric polynomials having degrees at most $d$  into isotypic components was studied in \cite{BC-imrn} and \cite{BC-focm} 
where several results are proved.
One important result is a severe restriction on the partitions that are allowed to appear in the isotypic 
decomposition of the cohomology -- which cuts down the possibilities for the allowed partitions  \emph{from exponential  to polynomial} (for fixed $d$).
The following theorem is a slightly simplified version of Theorem 4 in \cite{BC-focm} and will be used in the proof of our new stability result
(Theorem~\ref{thm:main} below).

\begin{theorem}\cite{BC-focm}
\label{thm:restriction}
Let $d \geq 2$,
and
$V \subset \R^n$ be a real variety defined by symmetric polynomials
of degree bounded by $d$. 
Then, for  all $\lambda \vdash n$, 
if $m_{i, \lambda}(V) > 0$, then
\[
\length(\lambda) <   i+2d-1.
\]
\end{theorem} 

Independent of the above results,
the phenomenon of representational and homological stability (see for example 
\cite{Church-et-al})  is an  active topic of research in algebraic topology.
One basic phenomenon of (homological) stability that motivates this study is 
the fact that for any fixed $p$, 
and any manifold $X$, 
$b_p(C_n(X))$, where $C_n(X)$ is the ordered $n$-th configuration space of $X$, 
is eventually given by a polynomial in $n$. The space $C_n(X)$ admits
an $\mathfrak{S}_n$ action which induces an $\mathfrak{S}_n$-module structure on $\H_p(C_n(X))$. The homological stability is then a 
consequence of the stability of the multiplicities of certain 
Specht modules in $\H_p(C_n(X))$ for large $n$. All the above can be put in a much broader context of the category of $\mathrm{FI}$-modules. However, we do not need this generality for the 
application that we discuss below.

Inspired by the representational stability phenomenon,  the following conjecture was made in \cite{BC-imrn} about the \emph{growth rate} 
of the multiplicities of the Specht modules in the
cohomology modules of certain natural sequences of symmetric  semi-algebraic sets.  
We state this conjecture below. 
But in order to do so we first need to introduce some definitions. 

We let
\[
\Lambda_n = \mathrm{R}[X_1,\ldots,X_n]^{\mathfrak{S}_n}
\]
denote the graded ring of invariant polynomials, with natural graded homomorphisms
$
\Lambda_{n+m} \rightarrow \Lambda_n
$
obtained
by setting $X_{n+m}, \ldots, X_{n+1}$ to $0$. 
We denote by 
\[\Lambda = \projlim  \Lambda_n
\]
(where the limit is taken in the 
category of graded rings), 
and denote by 
\[
\phi_n: \Lambda \rightarrow \Lambda_n
\]
the graded homomorphisms
induced by the limit (see \cite[pages 18-19]{Macdonald-book}).

By a standard abuse of notation, after dropping $n$ from the subscript,
we will consider the symmetric polynomials $\sigma_\ell, N_\ell$
(see Notation~\ref{not:elementary} and Notation~\ref{not:Newton}) as
elements of the ring $\Lambda$. 

More precisely, for every $\ell, n \geq 0$,
we have
\begin{eqnarray*}
\phi_n(\sigma_{\ell}) &=& \sigma_{\ell,n}, \\
\phi_n(N_{\ell}) &=& N_{\ell,n}.
\end{eqnarray*}

Now, suppose $I = (f_1,\ldots,f_k) $ is a finitely generated ideal of $\Lambda$. 
Then,  $I$ defines in a
natural way symmetric real algebraic sets  
\[
V_n(I)  = \mathrm{Zer}(\phi_n(f_1),\ldots,\phi_n(f_k)) \subset \mathrm{R}^n, n > 0.
\]

For any fixed partition $\lambda = (\lambda_1,\ldots,\lambda_{\ell}) \vdash d $, 
we denote for $n \geq \lambda_1 + d $ 
\[\{\lambda\}_n = (n-d, \lambda_1,\ldots,\lambda_\ell).
\]
(Note that the above 
definition of the sequence of partitions 
$$
(\{\lambda = \left(\lambda_1,\ldots) \vdash d\}_n\right)_{n \geq \lambda_1 + d}
$$
is standard in the asymptotic study 
of representations of $\mathfrak{S}_n$ as $n \rightarrow \infty$
(see for example \cite[Eqn. (6.3.1)]{Deligne2004}).)

We are now in a position to state the 
the conjecture made in \cite{BC-imrn}.

\begin{conjecture} \cite{BC-imrn}
\label{conj:main}
For any fixed $p \geq 0$,  
$m_{p, \{\lambda\}_n}(V_n(I))$ (see \eqref{eqn:def:multiplicity} for definition)
is eventually given by a polynomial in $n$.
\end{conjecture}

The evidence in favor of Conjecture~\ref{conj:main} is a little sparse.
It was verified in the following
very special case in \cite{BC-imrn}. 

Let $\lambda = (\lambda_1,\ldots,\lambda_\ell) \vdash d$,
and $I \subset \Lambda$ the ideal generated by the symmetric function
$N_4 - 2N_3 + N_2 \in \Lambda$. In this case,
for which for each $n >0$, 
the corresponding real algebraic set  $V_n(I)$  equals $\mathcal{B}_n = \{0,1\}^n$.
We have for all large enough $n$ (see \cite[Remark 5.3]{BC-imrn}),
\begin{equation}\label{eqn:BC-imrn}
m_{i,\{\lambda\}_n}(V_n(I)) =
\left \{ \begin{array}{ll}
 n - 2\lambda_1 + 1, & \mbox{if $i=0$ and $\length(\lambda) \leq 1$,} \\
 0 & \mbox{otherwise.}
\end{array}
\right.
\end{equation}

Notice that the right hand side of Eqn. \eqref{eqn:BC-imrn} is a polynomial in $n$ for any fixed $\lambda$.

In this paper we verify Conjecture~\ref{conj:main} for 
an infinite class of ideals. Instead of considering just the 
ideal generated by a very particular linear combination of \emph{Newton
symmetric functions} as above, we are able to handle all principal 
ideals in $\Lambda$
which are generated by \emph{arbitrary} linear combinations of 
the \emph{elementary symmetric functions}.

We  prove the following theorem.

\begin{theorem} 
\label{thm:main}
Let $f = \sum_{i=0}^{d} a_i \sigma_i \in \Lambda$ be a linear combination of the
elementary symmetric functions $\sigma_i \in \Lambda,   0 \leq i \leq d$,  and 
let $I = (f)$.   Then,  for any  partition $\lambda$ and $n$ large enough,
$m_{0, \{\lambda\}_n}(V_n(I))$
equals $0$ if
$\mathrm{length}(\lambda) > 0$,   and stablizes to a (possibly non-zero) constant if
 $\mathrm{length}(\lambda) = 0$ 
 (i.e when $\lambda$ is the empty partition).
 \end{theorem}

\begin{remark}
\label{rem:main}
Theorem~\ref{thm:main} verifies Conjecture~\ref{conj:main}
for ideals generated by one linear combination of elementary symmetric functions, 
with $p=0$. Note that in comparison to the special case of Conjecture~\ref{conj:main} proved in \cite{BC-imrn}, the family of ideals
that we are able to handle (while still being principal) is considerably larger.
It should also be noted that Theorem~\ref{thm:main} proves a strong form of
Conjecture~\ref{conj:main} for the principal ideals that we consider in this paper -- in that the multiplicities of the Specht modules corresponding to $\{\lambda\}_n$ actually stabilize to a constant
(not just a polynomial in $n$). In general such a strong version
of Conjecture~\ref{conj:main} cannot hold as exhibited in
Eqn. \eqref{eqn:BC-imrn}.
\end{remark} 

\begin{remark}
Note that the limit
\[
\lim_{n \rightarrow \infty} m_{0, \{\lambda\}_n}(V_n(I))
\]
which exists by Theorem~\ref{thm:main}
can be strictly bigger than $1$.
For instance, we will show at the end of Section \ref{subsec:proof:thm:connected} that if $f = \sigma_2 - 1, g = \sigma_3 - \sigma_1$,
$I = (f), J = (g), \lambda= ()$, then
\begin{eqnarray*}
\lim_{n \rightarrow \infty} m_{0, \{\lambda\}_n}(V_n(I)) &=& 2, \\
\lim_{n \rightarrow \infty} m_{0, \{\lambda\}_n}(V_n(J)) &=& 3.
\end{eqnarray*}
\end{remark}

The rest of the paper is devoted to the proofs of the theorems stated 
above.

\section{Proofs of the main results}
Even though theorems that we have stated in the previous section were formulated over an arbitrary real closed field $\R$, using a standard
application of the Tarski-Seidenberg transfer principle
(see for example \cite[Theorem 2.80]{BPRbook2}) it suffices to
prove them for $\R = \mathbb{R}$. In the rest of the paper we will assume
$\R = \mathbb{R}$ so that we are free to use certain basic results 
(such as existence of Leray spectral sequence, proper base change theorem 
etc.) 
without having to formulate these over arbitrary real closed fields.

\subsection{Proof of Theorem~\ref{thm:ccez}} The idea of the proof is as follows. Let $P \in \R[X_1, \dots, X_n]$ be a multi-affine polynomial of degree $d \in \N$. Suppose  
$$
P(X_1, \dots, X_n) = X_nQ(X_1, \dots, X_{n-1}) + R(X_1, \dots, X_{n-1})
$$ 
with $Q, R$ multi-affine,  $Q\ne 0$ and $\deg Q = d-1$.
The main point of the proof is to show that there is no semi-algebraic connected
component of $Z(P, \R^n)$ included in $Z(Q, \R^n)$. Once this is done, $b_0(Z(P, \R^n))$ is bounded by the number of semi-algebraic connected components of the set $Z(P, \R^n) \cap (\R^n \setminus 
Z(Q, \R^n))$. Finally, we bound this last number using Theorem \ref{thm:ccdz}, which we prove first and might be of independent interest.

\begin{theorem}
\label{thm:ccdz}
Let $P \in \R[X_1, \dots, X_n]$ be a multi-affine polynomial of degree $d \in \N_0$. The number of  semi-algebraically  connected components
of $\R^n \setminus \Z(P, \R^n)$ is bounded by $2^{d}$.
\end{theorem}

\begin{proof} The proof is by induction on $d$. The result is clear for $d = 0$ and $d = 1$. Suppose now $d \ge 2$ and 
$$
P(X_1, \dots, X_n) = X_nQ(X_1, \dots, X_{n-1}) + R(X_1, \dots, X_{n-1})
$$ 
with $Q, R$ multi-affine. Without loss of generality we suppose $Q\ne 0$ and $\deg Q = d-1$.
Since every semi-algebraically  connected component of $\R^n \setminus \Z(P, \R^n)$ 
intersects $\R^n \setminus \Z(Q, \R^n)$, the number of 
semi-algebraically
connected components of $\R^n \setminus \Z(P, \R^n)$ is bounded by the number of 
semi-algebraically
connected components of 
$$
\displaylines{
\left(\R^n \setminus \Z(P, \R^n)\right) \cap 
\left(\R^n \setminus \Z(Q, \R^n))\right) = \cr
\left\{(x_1, \dots x_n) \in \R^n \ | \ Q(x_1, \dots, x_{n-1}) \ne 0, \ x_n \ne \frac{-R(x_1, \dots, x_{n-1})}{Q(x_1, \dots, x_{n-1})}\right\},
}
$$
which is twice the number of 
semi-algebraically
connected components of  
$\R^{n} \setminus Z(Q, \R^{n})$, or equivalently, of 
$\R^{n-1} \setminus Z(Q, \R^{n-1})$. We conclude using the inductive hypothesis.
\end{proof}

\begin{proof}[Proof of Theorem~\ref{thm:ccez}] 
We will denote by $e_1, \dots, e_n$ the elements of the 
standard
basis of $\R^n$, and denote by $\langle e_i \rangle$ the span of $e_i$.

We consider first the case of $P$ reducible in $\R[X_1, \dots, X_n]$. Suppose without loss of generality
$P = P_1P_2$ with $P_1$ and $P_2$ non-constant multi-affine polynomials, $P_1 \in \R[X_1, \dots, X_{m}]$ and $P_2 \in \R[X_{m+1}, \dots, X_{n}]$. 
If $C_1, \dots,  C_\ell$ and 
$D_1, \dots, D_{\ell'}$
are the semi-algebraic connected components of the non-empty sets
$\Z(P_1, \R^m)$ and  $\Z(P_2, \R^{n-m})$ respectivelly, then
$$
\Z(P, \R^n) = (C_1 \times \R^{n-m}) \cup \dots \cup (C_\ell \times \R^{n-m}) \cup ( \R^{m} \times D_1) \cup \dots \cup ( \R^{m} \times D_{\ell'}),
$$
which is semi-algebraically connected.
Therefore,  $b_0(Z(P, \R^n)) = 1$.

Now we consider the case of $P$ irreducible in $\R[X_1, \dots, X_n]$. Suppose 
$$P(X_1, \dots, X_n) = X_nQ(X_1, \dots, X_{n-1}) + R(X_1, \dots, X_{n-1})$$
with $Q, R$ multi-affine. Without loss of generality suppose $Q \ne 0$
and $\deg Q = d-1$.
We will prove that there is no connected component of $Z(P, \R^n)$ included in $Z(Q, \R^n)$.
If $n = 1$ then $Z(Q, \R^n) = \emptyset$ and we are done.
From now on we consider $n \ge 2$. 

For every $x = (x_1, \dots, x_n) \in \R^n$ we denote $\bar x = (x_1, \dots, x_{n-1}) \in \R^{n-1}$ and
$\tilde x = (x_1, \dots, x_{n-2}) \in \R^{n-2}$.

Suppose that $C$ is a connected component of 
$Z(P, \R^n)$ included in $Z(Q, \R^n)$ and
take $z = (z_1, \dots, z_n) \in C$. Since $P(z) = Q(\bar z) = 0$, then $R(\bar z) = 0$ and  
$P(\bar z, z') = 0$ for every $z' \in \R$.
Since the line $z + \langle e_n \rangle$ is
semi-algebraically
connected, $C$ includes the line  $z + \langle e_n \rangle$. Take $\emptyset \ne I \subset \{1, \dots n\}$ of maximum cardinality such that
$$
z + \langle e_i \ | \ i \in I  \rangle \subset C.
$$
Notice that $\# I \le n-1$ since $P \not \equiv 0$, and $n \in I$ since 
\[C \subset Z(P, \R^n) \cap Z(Q, \R^n).
\]
Without loss of generality suppose $n-1 \not \in I$ and
$$Q(X_1, \dots, X_{n-1}) = X_{n-1}S(X_1, \dots, X_{n-2}) + T(X_1, \dots, X_{n-2}),$$
$$R(X_1, \dots, X_{n-1}) = X_{n-1}U(X_1, \dots, X_{n-2}) + V(X_1, \dots, X_{n-2}),$$
with $S, T, U, V$ multi-affine. 
We consider the following cases:
\begin{itemize}

\item For every $y \in z + \langle e_i \ | \ i \in I \rangle$,  $S(\tilde y) = U(\tilde y) = 0$:
We will prove that $$z + \langle e_i \ | \ i \in I \cup \{n-1\}  \rangle \subset C,$$
which is impossible since this contradicts the maximality of $I$.

Since $$z + \langle e_i \ | \ i \in I \cup \{n-1\}  \rangle = \cup_{y \in z + \langle e_i \ | \ i \in I  \rangle} (y + \langle e_{n-1} \rangle),$$
it is enough to prove that $y + \langle e_{n-1} \rangle \subset C$ for every $y \in z + \langle e_i \ | \ i \in I  \rangle$. Moreover, for any such $y$, since $y + \langle e_{n-1} \rangle$ is 
semi-algebraically
connected, it is enough to prove that
$y + \langle e_{n-1} \rangle \subset Z(P, \R^n)$.
Since $P(y) = Q(\bar y) = 0$, then $R(\bar y) = 0$.
Since in addition $S(\tilde y) = U(\tilde y) = 0$, 
$T(\tilde y) = V(\tilde y) = 0$.
Take any $w \in y + \langle e_{n-1} \rangle$, 
then
$Q(\bar w) = R(\bar w) = 0$ and $P(w) = 0$.

\item There exists 
\[
y \in z + \langle e_i \ | \ i \in I  \rangle
\]
such that $S(\tilde y) \ne 0$ and $U(\tilde y) = 0$:
We will prove  that the line $(\bar y, 0) + \langle e_{n-1} \rangle$ is included in $C$ and intersects $(\R^n \setminus Z(Q, \R^n))$, which is impossible
since this contradicts the fact that $C \subset Z(Q, \R^n)$.

To prove that $(\bar y, 0) + \langle e_{n-1} \rangle \subset C$, since $(\bar y, 0) + \langle e_{n-1} \rangle$ is semi-algebraically connected, it is enough to prove that $(\bar y, 0) + \langle e_{n-1} \rangle \subset Z(P, \R^n)$.
Since $P(y) = Q(\bar y) = 0$, then $R(\bar y) = 0$.
Since in addition $U(\tilde y) = 0$,  $V(\tilde y) = 0$. Take any $w \in (\bar y, 0) + \langle e_{n-1} \rangle$. Then $R(\bar w) = 0$ and
$P(w) = 0 \cdot Q(\bar w) + R(\bar w) = 0$. 

For $w \in (\bar y, 0) + \langle e_{n-1} \rangle$, 
$Q(w) = w_{n-1}S(\tilde y) + T(\tilde y) = 0$ 
if and only if $w_{n-1} = -T(\tilde y)/S(\tilde y)$.
It follows that $(\bar y, 0) + \langle e_{n-1} \rangle$ 
intersects $(\R^n \setminus Z(Q, \R^n))$.

\item There exists $y \in z + \langle e_i \ | \ i \in I  \rangle$ such that $U(\tilde y) \ne 0$:
We will prove that the polynomial $P$ is reducible, contradicting our assumption.

If $(\bar y,0)$ is in the closure of $Z(P, \R^n) \cap (\R^n \setminus Z(Q, \R^n))$, then it is in 
$Z(P, \R^n) \setminus C$; which is impossible since $(\bar y,0) \in C$. 
Hence there exists $\varepsilon > 0$ such that  
\[
\left( B(\tilde y, \varepsilon)
\times (y_{n-1} - \varepsilon, y_{n-1} + \varepsilon) \times 
(-\varepsilon, \varepsilon) \right) \ \cap \  
\left(
Z(P, \R^n) \cap (\R^n \setminus Z(Q, \R^n))
\right)
\]
is empty.
Moreover we can also suppose that $U$ does not vanish on $B(\tilde y, \varepsilon)$.

Since $P(y) = Q(\bar y) = 0$,  $R(\bar y) = 0$,
and since $U(\tilde y) \ne 0$,  $y_{n-1} = -V(\tilde y)/U(\tilde y)$.
This implies that $$\lim_{w \to \tilde y} -V(w)/U(w) = y_{n-1},$$
and that there exists $0 < \delta < \varepsilon$ such that $-V(w)/U(w) \in (y_{n-1} - \varepsilon, y_{n-1} + \varepsilon)$ for every
$w \in B(\tilde y, \delta)$.

For each $w\in  B(\tilde y, \delta)$, 
\[
(w, -V(w)/U(w), 0) \in B(\tilde y, \delta) \times (y_{n-1} - \varepsilon, y_{n-1} + \varepsilon) \times (-\varepsilon, \varepsilon) 
.
\]
Since $P(w, -V(w)/U(w), 0) = 0,$  we have that $Q(w, -V(w)/U(w)) = 0$, 
and we get
\[
\frac{-V(w)}{U(w)}S(w) + T(w) = 0,
\]
and  
\[
U(w)T(w) = V(w)S(w).
\]

Since this equality holds in the open set $B(\tilde y, \delta)$, we have 
\[
UT = VS \in R[X_1, \dots, X_{n-2}].
\]
Suppose
$U = U_1E$ and $V = V_1E$ with $E = \gcd(U, V) \in \R[X_1, \dots, X_{n-2}]$. Then $S = U_1F$ and $T = V_1F$ for some $F \in \R[X_1, \dots, X_{n-2}]$,
and $$P = X_n(X_{n-1}U_1F + V_1F) +
X_{n-1}U_1E + V_1E  = (X_nF + E)(X_{n-1}U_1 + V_1).
$$

\end{itemize}

After considering all the possible cases, we conclude that there is no semi-algebraic connected component of $Z(P, \R^n)$ included in $Z(Q, \R^n)$. This implies that $b_0(Z(P, \R^n))$ is bounded by the number of semi-algebraic connected components of the set
$$
\displaylines{
Z(P, \R^n) \cap  (\R^n \setminus Z(Q, \R^n)) = \cr \left\{(x_1, \dots x_n) \in \R^n \ | \ Q(x_1, \dots, x_{n-1}) \ne 0, x_n = \frac{-R(x_1, \dots, x_{n-1})}{Q(x_1, \dots, x_{n-1})}\right\},
}
$$
which equals the number of semi-algebraically connected components of the set $(\R^n \setminus Z(Q, \R^n))$. This number is bounded by $2^{d-1}$ by Theorem~\ref{thm:ccdz}.
\end{proof}

\subsection{Proof of Theorem~\ref{thm:example}}
\label{subsec:proof:thm:example}
For every $n >0$,  want to produce a symmetric multi-affine polynomial
$P \in \R[X_1,\ldots,X_n]$ of small degree (in fact we will take the degree to be $4$) having large Betti number (growing super-polynomially with $n$).
As mentioned earlier the usual trick of taking sum of squares does not work well with multi-affine polynomials. For example, the sequence of polynomials
\[
P_n = \sum_{i =1}^n X_i^2(X_i -1)^2
\]
has the property that each polynomial 
is symmetric, of degree $4$, having sum of Betti numbers equal to $2^n$ 
(and so growing exponentially with $n$), but $P_n$ is not multi-affine.

Therefore, we take an indirect approach. We leverage the fact that
the polynomials $P_1,P_2,P_3$ in Example~\ref{ex:three_pols}, being linear combinations of elementary symmetric polynomials, are each symmetric and multi-affine. Moreover, 
\[\H^0(\Z(\{P_1,P_2,P_3\},\R^n))
\]
as a 
$\mathfrak{S}_n$-module is easy to understand and has a Specht module
occurring in it of large dimension. 

We prove (Proposition~\ref{prop:example:specht} below) using a
spectral sequence argument that each Specht module that appears in
$\H^0(\Z(\{P_1,P_2,P_3\},\R^n))$
must appear
in at least one of the cohomology modules 
$\H^0(\Z(P,\R^n)),\ldots,\H^5(\Z(P,\R^n))$ for some $P$ in the linear span of $P_1,P_2,P_3$. 

Proposition~\ref{prop:example:specht} follows from a more general
result (Proposition~\ref{prop:example} below). 
Proposition~\ref{prop:example} relates the 
vanishing of the multiplicities of a Specht module in the low
dimensional (up to dimension $2p-1$ for some $p > 0$) 
cohomology modules of the hypersurfaces defined by 
symmetric polynomials in any linear subspace of symmetric polynomials,
to the vanishing of the same Specht module in the \emph{zero-th} cohomology of
the \emph{intersections} of at most $p$ of such hypersurfaces.

The key idea here is that if a finite group acts on the stalks of a constructible sheaf and the isotypic component corresponding to a certain irreducible representation is zero at all stalks, then the isotypic component of that irreducible occurs with zero multiplicity in the cohomology of that sheaf (see Claim~\ref{claim:proof:prop:example:5} in the proof of Proposition~\ref{prop:example} below).

\begin{proposition}
\label{prop:example}
Let $\lambda \vdash n, \lambda \neq (n)$, $p > 0$,  
$L \subset \R[X_1,\ldots,X_n]^{\mathfrak{S}_n}$ a linear subspace
of the vector space of symmetric polynomials, and $B \subset \R^n$
a symmetric, closed and bounded semi-algebraic set.

Suppose that for all $P \in L$ and $0 \leq i \leq 2p-1$, 
\begin{equation}
\label{eqn:prop:example:1}
m_{i,\lambda}(\Z(P,B)) = 0   
\end{equation}
(cf. Eqn.~\eqref{eqn:def:multiplicity}).

Then, for all $q,  1 \leq q\leq p$,  and  $P_1,\ldots,P_q \in L$,
\[
m_{0,\lambda}(\Z(\{P_1,\ldots,P_q\},B)) = 0.
\]
\end{proposition}

We will use the following lemma in the proof of Proposition~\ref{prop:example}. It is an equivariant of a similar inequality that appears in \cite[Proposition 7.33 (b)]{BPRbook2}.
\begin{lemma}
\label{lem:proof:prop:example:6}
Suppose that $V_1,\ldots,V_m$ be symmetric closed semi-algebraic subsets of $\R^n$.
For $J \subset [1,m]$ denote 
\[V^J = \bigcup_{j \in J} V_j, V_J = \bigcap_{j \in J} V_j.
\]
Then for $i \geq 0$ and 
$\lambda \vdash n$,
\begin{equation}
\label{eqn:lem:proof:prop:example:6:1}
m_{i,\lambda}(V_{[1,m]}) \leq \sum_{j=1}^{n-i} \sum_{J \subset [1,m], \card(J) = j} m_{i+j-1,\lambda}(V^J).
\end{equation}
\end{lemma}

\begin{proof}
The proof uses Schur's lemma and an $\mathfrak{S}_n$-equivariant version of the proof of a similar inequality in the non-symmetric case in \cite[Proposition 7.33 (b)]{BPRbook2}.
We first observe that claim is obviously true when $m = 1$. 

The claim is now proved by induction on $m$. Assume that the induction
hypothesis holds for all $m - 1$ closed, symmetric semi-algebraic subsets of $\R^n$,
and for all $i \geq 0$ and $\lambda \neq (n)$.

It follows from the standard Mayer-Vietoris sequence that there is an exact
sequence where each map is $\mathfrak{S}_n$-equivariant.
\[
\cdots \rightarrow
\H^i(V_{[1,m-1]}) \oplus \H^i(V_m) \rightarrow \H^i(V_{[1,m]}) \rightarrow \H^{i+1}(V_{[1,m-1]} \cup V_m) \rightarrow \cdots
\]

Using Schur's lemma and restricting to the isotypic component corresponding
to $\mathbb{S}^\lambda$ we obtain an exact sequence
\[
\cdots \rightarrow
\H^i(V_{[1,m-1]})_\lambda \oplus \H^i(V_m)_\lambda \rightarrow \H^i(V_{[1,m]})_\lambda \rightarrow \H^{i+1}(V_{[1,m-1]} \cup V_m)_\lambda \rightarrow \cdots
\]
from which it follows that
\begin{equation}
  \label{7:eq:propb1} 
m_{i,\lambda} (V_{[1,m]}) \leq m_{i,\lambda} (V_{[1,m-1]}) + m_{i,\lambda} (V_m) + m_{i + 1,\lambda} (V_{[1,m-1]} \cup V_m).
\end{equation}

Applying the induction hypothesis to the closed symmetric semi-algebraic sets $V_1,\ldots,V_{m-1}$, we deduce that
\begin{eqnarray}
\label{7:eq:propb2}
  m_{i,\lambda} (V_{[1,m-1]}) & \leq & \sum_{j = 1}^{n - i}
  \sum_{J \subset [1,m-1], \card(J) = j}  m_{i + j - 1,\lambda} (V^J).
\end{eqnarray}
Next, applying the induction hypothesis to the closed symmetric semi-algebraic sets,
$V_1 \cup V_m, \ldots, V_{m-1} \cup V_m$ we obtain

\begin{eqnarray}
  \label{7:eq:propb3} m_{i + 1,\lambda} (V_{[1,m-1]} \cup V_m) & \leq &
  \sum_{j = 1}^{n - i - 1} 
  \sum_{
    J \subset [1,m-1], \card(J) = j}
m_{i + j,\lambda} (V^{J \cup \{m\}}).
 \end{eqnarray}

We obtain from inequalities \eqref{7:eq:propb1}, \eqref{7:eq:propb2}, and \eqref{7:eq:propb3} that

\[ m_{i,\lambda} (V_{[1,m]}) \leq \sum_{j = 1}^{n - i}
   \sum_{
     J \subset [1,m], \card(J) = j}
   m_{i + j - 1,\lambda} (V^J),
\]
which finishes the induction.
\end{proof}

\begin{proof}[Proof of Proposition~\ref{prop:example}]
We first prove a series of claims (Claims~\ref{claim:proof:prop:example:1}-\ref{claim:proof:prop:example:5} below).
In these claims we will use the following notation.
Let $\underline{P} = (P_1,\ldots,P_q) \in L^q$ 
for some $q \geq 1$, and we denote 
\[
\Omega= \{ \omega= (\omega_1,\ldots,\omega_q) \in \Sphere^{q-1} \;\vert \; \omega_1 \geq 0, \ldots,\omega_q \geq 0\},
\]
where $\Sphere^{q-1}$ denotes the unit sphere in $\R^q$.

Following a technique introduced by Agrachev \cite{Agrachev,Ag},
for $\omega \in \Omega$ we denote
\[
\omega \underline{P} = \omega_1 P_1 + \cdots + \omega_q P_q, 
\]
and denote
\[
S(\underline{P},B) = \{(\omega, x) \in \Omega \times B \;\vert\; \omega\underline{P}(x) \leq 0\}. 
\]

We denote by $\pi_1: S(\underline{P},B) \rightarrow \Omega$ and 
$\pi_2: S(\underline{P},B) \rightarrow B$ the restrictions to $S(\underline{P},B)$ of the projection maps $\Omega \times B \rightarrow \Omega$ and $\Omega \times B \rightarrow B$ respectively.

\begin{claim}
\label{claim:proof:prop:example:1}
\[
\pi_2(S(\underline{P},B)) = \{x \in B \;\vert\; \bigvee_{j=1}^{q} (P_j(x) \leq 0) \}.
\]
\end{claim}
\begin{proof}[Proof of Claim~\ref{claim:proof:prop:example:1}]
Suppose that $P_j(x) \leq 0$, with $1 \leq j \leq q$. Let
\[
\omega^{(j)} = (\delta_{1,j},\ldots, \delta_{q,j})  \in \Omega.
\]

Then clearly $\omega^{(j)} \underline{P}(x) \leq 0$ and hence
$(\omega^{(j)},x) \in S(\underline{P},B)$, proving that $x \in \pi_2(S(\underline{P},B))$.

Conversely, if $x \in \pi_2(S(\underline{P},B))$, then there exists 
$\omega \in \Omega$, such that $\omega \underline{P}(x) \leq 0$. 
If $P_j(x) > 0$ for every $j, 1 \leq j \leq q$, then $\omega \underline{P}(x) >0$, since $\omega$ has at least one coordinate not equal to $0$ and hence strictly positive. This is a contradiction. So there exists $j, 1 \leq j \leq q$, such that $P_j(x) \leq 0$.

This completes the proof of the claim.
\end{proof}

\begin{claim}
\label{claim:proof:prop:example:2}
The map $\pi_2$ induces an isomorphism of $\mathfrak{S}_n$-modules
\[
\H_*(S(\underline{P},B)) \rightarrow \H_*(\pi_2(S(\underline{P},B))).
\]
\end{claim}

\begin{proof}[Proof of Claim~\ref{claim:proof:prop:example:2}]
The map $\pi_2$ is clearly $\mathfrak{S}_n$-equivariant. For $x \in \pi_2(S(\underline{P},B))$, the fiber $\pi_2^{-1}(x)$ is a non-empty intersection of the sphere $\Sphere^{q-1}$ with the 
polyhedral cone defined by the linear inequalities,
\[
\omega_1 \geq 0, \ldots, \omega_q \geq 0, \omega_1 P_1(x) + \cdots \omega_q P_q(s) \leq 0,
\]
and hence is contractible. This implies that the induced
map $\pi_{2,*}: \H_*(S(\underline{P},B)) \rightarrow \H_*(\pi_2(S(\underline{P},B)))$ is an isomorphism by the Vietoris-Begle theorem \cite[page 344]{Spanier}.
\end{proof}

\begin{claim}
\label{claim:proof:prop:example:3}
Eqn. \eqref{eqn:prop:example:1} implies that for all $P \in L$, $0 \leq i \leq 2p-1$,
\[
m_{i,\lambda}(\RR(P \leq 0,B)) = 0.
\]
\end{claim}
\begin{proof}[Proof of Claim~\ref{claim:proof:prop:example:3}]
The Mayer-Vietoris exact sequence in homology yields the following exact sequence relating the homology groups of $\RR(P \leq 0,B), \RR(P \geq 0,B), \RR(P=0,B)$:
\[
\cdots \rightarrow \H_i(\RR(P=0,B)) \rightarrow \H_i(\RR(P \leq 0,B)) \oplus \H_i(\RR(P \geq 0,B)) \rightarrow \H_i(B) \rightarrow \cdots
\]

Note that each arrow in the above sequence represents an homomorphism of $\mathfrak{S}_n$-modules. Thus, by Schur's lemma they restrict to give an exact sequence between the $\mathbb{S}^\lambda$-isotypic components. Noticing that $\lambda \neq (n)$, and hence $m_{0,\lambda}(B) = 0$,
we obtain the inequality for each $i \geq 0$,
\begin{equation}
\label{eqn:claim:proof:prop:example:3:1}
m_{i,\lambda}(\RR(P \geq 0,B)) + m_{i,\lambda}(\RR(P \leq 0,B)) \leq 
m_{i,\lambda}(\RR(P = 0,B)).
\end{equation}

This together with \eqref{eqn:prop:example:1} implies that for 
$0 \leq i \leq 2p-1$, 
\begin{equation}
\label{eqn:claim:proof:prop:example:3:2}
m_{i,\lambda}(\RR(P \leq 0,B)) 
= 0.
\end{equation}
The claim follows from \eqref{eqn:claim:proof:prop:example:3:1} and \eqref{eqn:claim:proof:prop:example:3:2}.
 \end{proof}

\begin{claim}
\label{claim:proof:prop:example:4}
For each $\omega \in \Omega$, and $0 \leq i \leq 2p-1$, 
\[
m_{i,\lambda}(\pi_1^{-1}(\omega)) = 0.
\]
\end{claim}

\begin{proof}[Proof of Claim~\ref{claim:proof:prop:example:4}]
Follows immediately from Claim~\ref{claim:proof:prop:example:3} noting that
$
\pi_1^{-1}(\omega)
$
is equivariantly homeomorphic to $\RR(\omega \underline{P}\leq 0,B)$, 
and 
\[
\omega\underline{P} \in \mathrm{span}(P_1,\ldots,P_q) \subset L.
\]
\end{proof}

\begin{claim}
\label{claim:proof:prop:example:5}
For $0 \leq i \leq 2p-1$, 
\[
m_{i,\lambda}(S(\underline{P},B)) = 0.
\]
\end{claim}

\begin{proof}[Proof of Claim~\ref{claim:proof:prop:example:5}]
Let $S = S(\underline{P},B)$.
There exists a first-quadrant spectral sequence, $E_r^{s,t}$ (the Leray spectral sequence of the map $\pi_1$), 
converging to $\H^{s+t}(S)$, whose $E_2$-term is given by
\[
E_2^{s,t} = \H^s (\Omega, R^t \pi_{1*}(\Q_S)),
\] 
where $\Q_S$ denotes the constant $\Q$-sheaf on $S$.
The sheaf $R^t \pi_{1*}(\Q_S)$ is the sheaf associated to the 
presheaf which associates to every open subset $U \subset \Omega$,
the $\Q$-vector space, 
\[
\H^t(\pi_1^{-1}(U))
\]
(see \cite[Chapter II, Proposition 5.11]{Iversen}).
The set $\pi_1^{-1}(U)$ is stable under the action of $\mathfrak{S}_n$,
and so there exists an isotypic decomposition
\[
\H^t(\pi_1^{-1}(U)) \cong_{\mathfrak{S}_n} \bigoplus_{\mu \vdash n}
\left( \H^t(\pi_1^{-1}(U))\right)_\mu
\]
(cf. Notation~\ref{not:isotypic}).
Moreover, since the restriction homomorphisms of this presheaf are all $\mathfrak{S}_n$-equivariant, it follows from
Schur's Lemma and the definition of the sheafification functor 
(see for instance \cite[page 85]{Iversen})
that there is a direct sum decomposition of the sheaf $R^t \pi_{1*}(\Q_S))$
into its isotypic components 
$R^t \pi_{1*}(\Q_S))_\mu, \mu \vdash n$.

Thus, we have
\begin{eqnarray*}
R^t \pi_{1*}(\Q_S) &\cong & 
 \bigoplus_{\mu \vdash n} (R^t \pi_{1*}(\Q_S))_\mu.
\end{eqnarray*}

Since,  $\pi_1:S \rightarrow \Omega$ is a proper map,
using the proper base change theorem (see for example \cite[\S 3, Theorem 6.2]{Iversen}) we obtain
that for $\omega \in \Omega$,
\begin{equation*}
R^t \pi_{1*} (\Q_S)_\omega \cong \H^t(\pi_1^{-1}(\omega)),
\end{equation*}
and for $\mu\vdash n$,
\begin{equation}
\label{eqn:stalk}
(R^t \pi_{1*} (\Q_S))_\mu)_\omega \cong \H^t(\pi_1^{-1}(\omega))_\mu.
\end{equation}

Using Claim~\ref{claim:proof:prop:example:4} we have that 
for each $\omega \in \Omega$, and $0 \leq i \leq 2p-1$, 
\[
m_{i,\lambda}(\pi_1^{-1}(\omega)) = 0.
\]
Taking $\mu = \lambda$ in Eqn. \ref{eqn:stalk}, 
we have  $0 \leq t \leq 2p-1$,
\[
(R^t \pi_{1*} (\Q_S))_\lambda)_\omega \cong \H^t(\pi_1^{-1}(\omega))_\lambda
=0,
\]
which in turn implies that 
\begin{equation}
\label{eqn:stalk2}
R^t \pi_{1*} (\Q_S))_\lambda = 0.
\end{equation}

Now, 
\begin{eqnarray*}
E_2^{s,t} &\cong_{\mathfrak{S}_n}&  \H^s (\Omega, R^t \pi_{1*}(\Q_S)) \\
&\cong_{\mathfrak{S}_n}& \H^s (\Omega, \bigoplus_{\mu \vdash n} (R^t \pi_{1*}(\Q_S))_\mu)\\
&\cong_{\mathfrak{S}_n}&
\bigoplus_{\mu \vdash n} \H^s (\Omega,  (R^t \pi_{1*}(\Q_S))_\mu) \\
&=& \bigoplus_{\mu \vdash n} (E_2^{s,t})_\mu,
\end{eqnarray*}
where 
\[
(E_2^{s,t})_\mu = \H^s (\Omega, (R^t \pi_{1*}(\Q_S))_\mu).
\]

The differentials $d_r:E_r^{s,t} \rightarrow E_r^{s+r, t-r+1}$ in the spectral sequence $E_r^{s,t}$  are 
$\mathfrak{S}_n$-equivariant,  
and for each $\mu \vdash n$ using Schur's lemma yet again,
we have  for 
$r \geq 2$, 
$(E_r^{s,t})_\mu$ is a subquotient of $(E_2^{s,t})_\mu$.

It follows from the above and  Eqn. \eqref{eqn:stalk2} that
for $0 \leq t \leq 2p-1$, and all $s \geq 0$ and $r \geq 2$,
\[
(E_r^{s,t})_\lambda = 0.
\]
This implies that for all $i, 0 \leq i \leq 2p-1$,
\[
(\H^{i}(S))_\lambda = \bigoplus_{s + t = i} (E_\infty^{s,t})_\lambda = 0,
\]
or equivalently,
\begin{eqnarray*}
m_{i,\lambda}(S)&=&  0
\mbox{ for } 0 \leq i \leq 2p-1.
\end{eqnarray*}
\end{proof}

Observe that Claims~\ref{claim:proof:prop:example:2} and \ref{claim:proof:prop:example:5} together imply that 
for any
$\underline{P} = (P_1,\ldots,P_q) \in L^q, q \geq 1$,
and $0 \leq i \leq 2p-1$
\begin{equation*}
 m_{i,\lambda}(\pi_2(S(\underline{P},R))) = 0.   
\end{equation*}

Rewriting the above equation using Claim~\ref{claim:proof:prop:example:1}
we obtain that for $0 \leq i \leq 2p-1$
\begin{equation}
\label{eqn:proof:prop:example:1}
    m_{i,\lambda}(\RR(\bigvee_{j=1}^{q} (P_j \leq 0, 
    B
    ))) = 0.
\end{equation}

We are now in a position to finish the proof of Proposition~\ref{prop:example}.

We now fix $\underline{P} = (P_1,\ldots,P_q) \in L^q$, and 
assume that $1 \leq q \leq p$.
Observe that
\[
\RR(\bigwedge_{j=1}^{q} P_j =0, B) = 
\RR(\bigwedge_{j=1}^{q} ((P_j \leq 0) \wedge (-P_j \leq 0)),
B).
\]

Let
\begin{eqnarray*}
V_j &=& \RR(P_j \leq 0,B), \mbox{ for } j=1,\ldots,q, \\
V_j &=& \RR(-P_{j-q} \leq 0,
 B), \mbox{ for }j=q+1,\ldots,2q.
\end{eqnarray*}

Now 
Eqn.~\eqref{eqn:proof:prop:example:1}
applied to the various sub-tuples of the tuple
\[
(P_1,\ldots,P_q,-P_1,\ldots,-P_q) \in L^{2 q},
\]
implies taking $i=0$ that
for all $J \subset [1,2q]$, 
$m_{j -1,\lambda} (V^J) = 0$, where $j = \card(J)$
(noticing that 
$j-1 \leq 2p-1$, since $j = \card(J) \leq 2q \leq 2p$).
Inequality \eqref{eqn:lem:proof:prop:example:6:1} in Lemma~\ref{lem:proof:prop:example:6}  now implies that
\begin{eqnarray*}
m_{0,\lambda}(V_{[1,2q]}) &=& m_{0,\lambda} (\Z(\{P_1,\ldots,P_q\},
B)) \\
&=& 0.
\end{eqnarray*}
This finishes the proof of Proposition~\ref{prop:example}.
\end{proof}

\begin{proposition}
\label{prop:example:specht}
Let $B$ be a symmetric, closed, bounded symmetric semi-algebraic set containing
$\mathcal{B}_n$.
For $k > 0$, and $n \geq 2 k$, 
and each $\lambda = (n-j,j), 0 \leq j \leq k$,
there exists $P \in \Sigma_{4,n}$, such
that there exists $i, 0 \leq  i \leq 5$, 
\[
m_{i,\lambda}(\Z(P,B)) > 0.
\]
\end{proposition}

\begin{proof}
Following Example~\ref{ex:three_pols} we let
$$
\begin{array}{ccl}
P_1(X) & = & \sigma_{1,n}(X) - k, \\[2mm]
P_2(X) & = & \sigma_{2,n}(X) - \frac12 k(k-1), \\[2mm]
P_3(X) & = & (4k-6)\sigma_{3,n}(X) -4\sigma_{4}(X)  - \frac12k(k-1)^2(k-2). \\[2mm]
\end{array}
$$

Then, 
$\Z(\{P_1,P_2,P_3\},\R^{n})$ is equal to the subset of $\mathcal{B}_n= \{0,1\}^{n} \subset B$ of cardinality $\binom{n}{k}$
consisting of points with exactly $k$ $1$'s and $n-k$ $0$'s amongst its coordinates.

The $\mathfrak{S}_{n}$-module structure of 
$\H^0(\Z(\{P_1,P_2,P_3\},B))$
is well-studied.
It is isomorphic to the \emph{Young module}
$M^{(n-k,k)}$ \cite[page 139]{Ceccherini-book}
(see also \cite[Example 1.19]{BC-imrn}). 
\footnote{The Young module $M^{n-k,k}$ is  isomorphic to the induced module  $\Ind_{\mathfrak{S}_{k} \times \mathfrak{S}_{n-k}}^{\mathfrak{S}_n}
\mathbf{1}_{\mathfrak{S}_k} \boxtimes \mathbf{1}_{\mathfrak{S}_{n-k}}$.}
The isotypic decomposition of the Young module $M^{(n-k,k)}$ is given by 
\[
M^{(n-k,k)} \cong_{\mathfrak{S}_{n}} \bigoplus_{j=0}^{k} \mathbb{S}^{n-j, j}
\]
(see \cite[page 141, Eqn. (3.72)]{Ceccherini-book}).
Thus,
\begin{equation}
\label{eqn:proof:thm:example:2}
m_{0,\lambda}(\Z(\{P_1,P_2,P_3\},
B)) = 1 > 0,
\end{equation}
for  $\lambda = (n-j,j), 0 \leq j \leq k$.

Now suppose for the sake of contradiction that for all $P \in \Sigma_{4,n}$, and $\lambda = (n-j,j)$
\begin{equation}
\label{eqn:proof:thm:example:1}
    m_{i,\lambda}(\Z(P,
    B)) = 0, 
\end{equation}
for $0 \leq i \leq 5$.

But Eqns. \eqref{eqn:proof:thm:example:1} and
\eqref{eqn:proof:thm:example:2} together
contradict Proposition~\ref{prop:example} 
with $L = \Sigma_{4,n}$, and $p=q=3$.  
\end{proof}

In the proof of Theorem~\ref{thm:example}
we will also need the following lemma which is a straight-forward consequence of the hook formula.

\begin{lemma}
\label{lem:proof:thm:example:1}
For all $n \in \N$, 
and $\lambda = (n-\lfloor n/2 \rfloor, \lfloor n/2 \rfloor) \vdash n$
\begin{eqnarray}
\label{eqn:lem:proof:thm:example:1}
\dim \mathbb{S}^\lambda &=& \frac{1}{\lfloor n/2 \rfloor+1} \binom{n}{\lfloor n/2 \rfloor} \mbox{ if $n$ is even}, \\
\nonumber
&=& \frac{1}{2(\lfloor n/2 \rfloor+2)} \binom{n}{\lfloor n/2 \rfloor} 
\mbox{if $n$ is odd}.
\end{eqnarray}
In particular, there exists $c > 1$ such that for all $n > 1$,
\[
\dim \mathbb{S}^\lambda > c^n.
\]
\end{lemma}
\begin{proof}
Eqn. \eqref{eqn:lem:proof:thm:example:1} follows immediately 
from Eqn. \ref{eqn:hook} (hook length formula).
The last statement is a consequence of the inequality
\[
\frac{4^m}{2m+1} \leq \binom{2m}{m},
\]
which is valid for all $m > 0$.
\end{proof}

We are finally in a position to prove Theorem~\ref{thm:example}.
\begin{proof}[Proof of Theorem~\ref{thm:example}]
Since the set $\Sigma_{4,n}$ is invariant under scaling of variables,
we can assume without loss of generality that $B_n \supset \mathcal{B}_n$.
It follows from Proposition~\ref{prop:example:specht} that 
for $n \geq 2k$, and $\lambda = (n-k,k)$, there
exists $i, 0 \leq i \leq 5$ and $P \in \Sigma_{4,n}$ such that
\[
m_{i,\lambda}(\Z(P,
B_n)) > 0.
\]

It follows from Theorem~\ref{thm:ccez} and
Lemma~\ref{lem:proof:thm:example:1} that for each $k>0$ and 
$n$ large enough, 
$
m_{0,\lambda}(\Z(P,
B_n)) = 0,
$
for all $P \in \Sigma_{4,n}$
and $\lambda = (n-k,k)$.

So we get that that for each $k > 0$ and $n \geq 2k$ , 
there
exists $i, 1 \leq i \leq 5$ and $P \in \Sigma_{4,n}$ such that
\[
m_{i,\lambda}(\Z(P,
B_n)) > 0,
\]
with
$\lambda = (n-k,k)$.

Now choose $k = \lfloor n/2 \rfloor$ and use 
Lemma~\ref{lem:proof:thm:example:1}.
\end{proof}

\subsection{Proof of Theorem~\ref{thm:main}}\label{subsec:proof:thm:connected}
The proof is in two steps. 

We first prove (Proposition~\ref{prop:stable_con_comp}) that since the dimensions of the cohomology modules
$\H^0((\Z(\phi_n(f),\R^n))$ do not increase with $n$ (using Theorem~\ref{thm:ccez}), for $n$ large enough
they cannot have Specht modules in their isotypic decomposition which correspond to partitions that are not equal to the trivial partition
$(n)$ or its transpose $1^n$. We then use Theorem~\ref{thm:restriction} to rule out the partition $1^n$. This enables us to deduce that 
the $\H^0((\Z(\phi_n(f),\R^n))$ is a multiple of the trivial 
representation (i.e. $\H^0((\Z(\phi_n(f),\R^n)) = \H^0((\Z(\phi_n(f),\R^n))^{\mathfrak{S}_n}$) or equivalently that 
each semi-algeraically connected component of $\Z(\phi_n(f),\R^n)$ is stable under the  action of $\mathfrak{S}_n$. 

We next prove (Proposition~\ref{prop:main} below) using Proposition~\ref{prop:stable_con_comp} that the sequence of numbers
$(b_0(\Z(\phi_n(f),\R^n))_{n > 0}$ is non-increasing and so ultimately
constant.
Propositions~\ref{prop:stable_con_comp} and
Proposition~\ref{prop:main} together suffices to prove Theorem~\ref{thm:main}.

\begin{proposition}\label{prop:stable_con_comp}
Let $d,n \in \N$ with $d \geq 2, n 
> 2^{d-1} + 1$ and let 
$P \in \R[X_1, \dots, X_n]$ be a 
multi-affine symmetric polynomial with $\deg P = d$. 
Every semi-algebraic connected component of $\Z(P, \R^n)$ is stable
under the action of $\mathfrak{S}_n$. This is to say, for every semi-algebraic connected component $C$ of $Z(P, \R^n)$ and every $\alpha \in \mathfrak{S}_n$, 
$$
C = \{(z_{\alpha(1)}, \dots, z_{\alpha(n)}) \ | \ 
(z_1, \dots, z_n) \in C\}.
$$
\end{proposition}

\begin{proof}
Let $V = \Z(P,\R^n)$.
First observe that 
$
\H^0(V)^{\mathfrak{S_n}}
$
is isomorphic (as a vector space) to the isotypic component of the trivial representation $\mathbb{S}^{(n)}$ in $\H^0(V)$.
Second,
each semi-algebraically connected component of $V$ is stable under the
action of $\mathfrak{S}_n$ if and only if
\[
\H^0(V)^{\mathfrak{S}_n} = \H^0(V).
\]
Thus, it suffices to prove that 
$\H^0(V)$ is isomorphic as an $\mathfrak{S}_n$-module to a multiple of trivial representation which is the same as proving that
\[
m_{0,\lambda}(V) = 0,
\]
for $\lambda \neq (n)$.
Now it follows from Theorem~\ref{thm:ccez} that
\begin{eqnarray}
\nonumber
b_0(V) &=& \dim \H^0(V) \\
\nonumber
&=& \sum_{\lambda \vdash n} m_{0,\lambda}(V) \dim \mathbb{S}^\lambda \\
\label{eqn:proof:prop:stable_con_comp:1}
&\leq& 2^{d-1}.
\end{eqnarray}

It is an easy consequence of hook formula that 
\begin{equation}\label{eqn:proof:prop:stable_con_comp:2}
\dim S^{\lambda} = \left\{
\begin{array}{ll}
 1 & \mbox{ if $\lambda = (n), \, 1^n$} \\[1mm]
\geq n-1 & \mbox{otherwise}. 
\end{array}\right.
\end{equation}

Since, 
\[
n >  2^{d-1} +1,
\]
we have that
\[
n -1 \geq 2^{d-1} +1 >  b_0(V).
\]
It now follows from
\eqref{eqn:proof:prop:stable_con_comp:1} and
\eqref{eqn:proof:prop:stable_con_comp:2} that
\begin{equation}
\label{eqn:proof:prop:stable_con_comp:4}
 m_{0,\lambda}(V) = 0, \mbox{if $\lambda \neq (n),1^n$}.   
\end{equation}

However, since $d \geq 2$, and hence 
\[
\length(1^n) = n > 2^{d-1} + 1 
\geq 0 + 2d -1,
\]
it follows from Theorem~\ref{thm:restriction} that
\begin{equation}
\label{eqn:proof:prop:stable_con_comp:5}
 m_{0,\lambda}(V) = 0\ \mbox{if $\lambda = 1^n$}.   
\end{equation}

The proposition now follows from 
\eqref{eqn:proof:prop:stable_con_comp:4} and
\eqref{eqn:proof:prop:stable_con_comp:5}.
\end{proof}

\begin{lemma}\label{lem:every_con_comp_cuts}
Let $d,n \in \N$ with $n \ge 2^{d-1} + 1$ and let 
$P \in \R[X_1, \dots, X_n]$ be a 
multi-affine symmetric polynomial with $\deg P = d$. 
Every semi-algebraic connected component of 
$\Z(P, \R^n)$ intersects
the hyperplane $\Z(X_n, \R^n)$.
\end{lemma}

\begin{proof}
Suppose
$$
P(X_1, \dots, X_n) = X_n Q(X_1, \dots, X_{n-1}) + 
R(X_1, \dots, X_{n-1}), 
$$
$$
Q(X_1, \dots, X_{n-1}) = X_{n-1}S(X_1, \dots, X_{n-2}) + 
T(X_1, \dots, X_{n-2}), 
$$
with $Q, R, S, T$ multi-affine. 
Notice that $Q$ and $R$ are symmetric as elements
in $\R[X_1, \dots, X_{n-1}]$
and
$S$ and $T$ are symmetric as elements
in $\R[X_1, \dots, X_{n-2}]$.
Let $C$ be a semi-algebraic connected component of 
$\Z(P, \R^n)$. We consider the following cases:

\begin{itemize}

\item There exists $z = (z_1, \dots, z_n) \in C$ and $1 \le i \le n$ with $z_i = 0$:

In this case, $(z_1, \dots, z_{i-1}, z_n, z_{i+1}, \dots, z_{n-1}, 0)  \in C \cap \Z(X_n, \R^n)$
by Proposition \ref{prop:stable_con_comp}.

\item There exists $1 \le i < j \le n$ and $z = (z_1, \dots, z_n) \in C$ with $z_i$ and $z_j$ of opposite non-zero sign:

Without loss of generality suppose $z_i > 0$ and $z_j < 0$. By Proposition \ref{prop:stable_con_comp}, if we consider $z'$ which is obtained from $z$
by swapping coordinates $z_i$ and $z_j$, then $z'$ also lies in $C$. Since $C$ is semi-algebraically arc-connected, 
there exists $z'' = (z''_1, \dots, z''_n)$ in $C$ with $z''_i = 0$, and then we proceed as in the first case. 

\item There exists $z = (z_1, \dots, z_n) \in C$ and $1 \le i \le n$ with \[
Q(z_1, \dots, \widehat{z_i}, \dots, z_n) = 0:
\]
\footnote{Here and elsewhere $\widehat{\cdot}$ denotes omission.}
Since 
\begin{eqnarray*}
0 &=& P(z) \\
&=& 
z_iQ(z_1, \dots, \widehat{z_i}, \dots, z_n) + R(z_1, \dots, \widehat{z_i}, \dots, z_n),
\end{eqnarray*}
we have that 
$R(z_1, \dots, \widehat{z_i}, \dots, z_n) = 0$ 
and therefore 
the line $z + \langle e_i \rangle$
is included in $C$. In particular
$(z_1, \dots, z_{i-1}, 0, z_{i+1}, \dots, z_n) \in C$ and we proceed as in the first case.

\item $C \subset (0, +\infty)^n$ and for every $z = (z_1, \dots, z_n) \in C$ and $1 \le i \le n$, 
$Q(z_1, \dots, z_{i-1}, z_{i+1}, \dots, z_n) \ne 0$:

This assumption implies that for $1 \le i \le n$, the polynomial 
\[
Q(X_1, \dots, X_{i-1}, X_{i+1}, \dots, X_n)
\]
has a constant sign on $C$. 
Let us consider a fixed value of $1 \le i \le n$ and see that for every $1 \le j \le n$ with $j \ne i$, the polynomial 
$$
S(X_1, \dots, X_{i-1}, X_{i+1}, \dots, X_{j-1}, X_{j+1}, \dots, X_n)
$$ 
never vanishes on $C$, and it has the same sign as $Q(X_1, \dots, X_{i-1}, X_{i+1}, \dots, X_n)$. Notice that this implies that the sign of 
$Q(X_1, \dots, X_{i-1}, X_{i+1}, \dots, X_n)$ on $C$ is independent of $i$. 

\begin{itemize}
\item If there exists $z = (z_1, \dots, z_n) \in C$ such that 
\[
S(z_1, \dots, \widehat{z_i}, \dots, \widehat{z_j}, \dots, z_n) = 0,
\]
since
$$
\displaylines{
0 \ne Q(z_1, \dots, \widehat{z_i}, \dots, z_n) = \cr
z_jS(z_1, \dots, \widehat{z_i}, \dots, 
\widehat{z_j}, \dots, z_n) + T(z_1, \dots, \widehat{z_i}, \dots
\widehat{z_j}, \dots, z_n),
}
$$
we have that $T(z_1, \dots, \widehat{z_i}, \dots, 
\widehat{z_j}, \dots, z_n) \ne 0$ and therefore 
for every $t \in \R$,  
$$
Q(z_1, \dots, \widehat{z_i}, \dots, z_{j-1}, t, z_{j+1}, \dots  z_n) \ne 0. 
$$
This implies 
that for each $t \in \R$ the point
\[
\left(z_1, \dots, z_{i-1}, a_t
, z_{i+1}, \dots, z_{j-1}, t, z_{j+1}, \dots, z_n\right)
\]
where 
\[
a_t = \frac{-R(z_1, \dots, \widehat{z_i}, \dots, z_{j-1}, t, z_{j+1}, \dots  z_n)}
{Q(z_1, \dots, \widehat{z_i}, \dots, z_{j-1}, t, z_{j+1}, \dots  z_n)},
\]
belongs to $C$, 
which contradicts the fact that $C \subset (0, \infty)^n$.

\item If there exists $z = (z_1, \dots, z_n) \in C$ such that $S(z_1, \dots, \widehat{z_i}, \dots, 
\widehat{z_j}, \dots, z_n)$ has opposite sign to 
$Q(z_1, \dots, \widehat{z_i}, \dots, 
z_n)$, 
then for $t \le z_j$ we have that 

$$
\displaylines{
Q(z_1, \dots, \widehat{z_i}, \dots, z_{j-1}, t, z_{j+1}, \dots, z_n) = \cr
(t-z_j)S(z_1, \dots, \widehat{z_i}, \dots, 
\widehat{z_j}, \dots, z_n) + Q(z_1, \dots, \widehat{z_i}, \dots
z_{j-1}, z_j, z_{j+1}, \dots, z_n) 
}
$$

is different from zero since it has the same sign as 
\[
Q(z_1, \dots, \widehat{z_i}, \dots
z_{j-1}, z_j, z_{j+1}, \dots, z_n).
\]
As before, 
this implies that 
that for each $t \in (-\infty, z_j]$, the point
\[
\left(z_1, \dots, z_{i-1}, a_t
, z_{i+1}, \dots, z_{j-1}, t, z_{j+1}, \dots, z_n\right)
\]
where
\[
a_t = \frac{-R(z_1, \dots, \widehat{z_i}, \dots, z_{j-1}, t, z_{j+1}, \dots  z_n)}
{Q(z_1, \dots, \widehat{z_i}, \dots, z_{j-1}, t, z_{j+1}, \dots, z_n)}
\]
belongs to $C$,
which contradicts the fact that $C \subset (0, \infty)^n$. 
\end{itemize}

Now, let us prove that if $z = (z_1, \dots, z_n) \in C$ and we take $(z'_1, \dots, z'_{n-1}) \in \R^{n-1}$ with
$z'_1 \ge z_1, \dots, z'_{n-1} \ge z_{n-1}$, then $Q(z'_1, \dots, z'_{n-1}) \ne 0$ and 
$$
\left(z'_1, \dots, z'_{n-1}, \frac{-R(z'_1, \dots, z'_{n-1})}
{Q(z'_1, \dots, z'_{n-1})}\right) \in C.
$$
We proceed by induction. Suppose that we know already that for some $1 \le i \le n$, 
$Q(z'_1, \dots, z'_{i-1}, z_i, \dots, z_{n-1}) \ne 0$ and 
$$
\left(z'_1, \dots, z'_{i-1}, z_i, \dots, z_{n-1}, \frac{-R(z'_1, \dots, z'_{i-1}, z_i, \dots, z_{n-1})}
{Q(z'_1, \dots, z'_{i-1}, z_i, \dots, z_{n-1})}\right) \in C.
$$
Then, for $t \ge z_i$, 
$$
Q(z'_1, \dots, z'_{i-1}, t, z_{i+1}, \dots, z_{n-1}) = 
$$
$$
=
(t-z_i)S(z'_1, \dots, z'_{i-1}, z_{i+1}, \dots, z_{n-1}) + Q(z'_1, \dots, z'_{i-1}, z_i, z_{i+1}, \dots, z_{n-1})
$$
is different from zero since it has the same sign as 
$Q(z'_1, \dots, z'_{i-1}, z_i, z_{i+1}, \dots, x_{n-1})$.
This implies that 
for each $t \in [z_i, z'_i]$, the point
\[
\left(z'_1, \dots, z'_{i-1}, t, z_{i+1}, \dots, z_{n-1}, 
a_t
\right)
\]
where
\[
a_t = \frac{-R(z'_1, \dots, z'_{i-1}, t, z_{i+1}, \dots, z_{n-1})}
{Q(z'_1, \dots, z'_{i-1}, t, z_{i+1}, \dots, z_{n-1})}
\]
belongs to $C$.

Finally, take any $z = (z_1, \dots, z_n) \in C$. For every $t \ge 0$, $Q(z_1 + t, \dots, z_{n-1} + t) \ne 0$ 
and 
$$
\left(z_1 + t, \dots, z_{n-1} + t, 
\frac{-R(z_1 + t, \dots, z_{n-1} + t)}
{Q(z_1 + t, \dots, z_{n-1} + t)}
\right) \in C.
$$
This is impossible because since $P$ is symmetric and $n \ge 2^{d-1}+1 \ge d+1$, it can be easily seen that $\deg R = d$, $\deg Q = {d-1}$ and 
$$
\lim_{t \to +\infty}  \frac{-R(z_1 + t, \dots, z_{n-1} + t)}
{Q(z_1 + t, \dots, z_{n-1} + t)} = - \infty, 
$$
which contradicts the assumption that $C \subset (0, +\infty)^n$.

\item $C \subset (-\infty, 0)^n$ and for every $z = (z_1, \dots, z_n) \in C$ and $1 \le i \le n$, $Q(z_1, \dots, z_{i-1}, z_{i+1}, \dots, z_n) \ne 0$:

In this case we proceed as in the previous one. 

\end{itemize}

\end{proof}

From now on we consider fixed $d \in \N$ and $a_0, \dots, a_d \in \R$ with $a_d \ne 0$.
For $n \ge d$, let 
$$
P_n = \sum_{0 \le \ell \le d} a_\ell \sigma_{\ell, n} \in \R[X_1, \dots, X_n].
$$

\begin{proposition} 
\label{prop:main}
The sequence $(b_0(\Z(P_n, \R^n))_{n \ge d}$ is eventually decreasing, and therefore eventually constant. 
\end{proposition}

\begin{proof} By Lemma \ref{lem:every_con_comp_cuts}, if
$n \ge 2^{d-1} + 1$, 
every semi-algebraic connected component of $Z(P, \R^n)$ intersects the hyperplane $Z(X_n, \R^n)$. 
Since 
$P_n(X_1, \dots, X_{n-1}, 0) = P_{n-1}(X_1, \dots, X_{n-1})$,
we have that 
$$
b_0(\Z(P_{n-1}, \R^{n-1})) \ge b_0(\Z(P_{n}, \R^{n})).
$$
\end{proof}

\begin{proof}[Proof of Theorem~\ref{thm:main}]
Theorem~\ref{thm:main} follows  from Propositions~\ref{prop:stable_con_comp} and \ref{prop:main}.
\end{proof}

We finish this section by showing two examples of ideals $I \subset \Lambda$ 
such that
\[
\lim_{n \rightarrow \infty} m_{0, \{\lambda\}_n}(V_n(I)) > 1,
\]
for $\lambda = ()$. First, we include an auxiliary lemma. 

\begin{lemma}\label{lem:aux:fin} Let $n \ge 3$. For $x \in \R^n$ with $N_1(x) = 0$, 
$$N_3(x)^2 \le \frac{(n-2)^2}{n(n-1)}N_2(x)^3.$$
\end{lemma}
\begin{proof}
The inequality holds if $x = 0$. If $x \ne 0$, we take $R^2 = 
N_2(x)$ and then the inequality
can be checked using Lagrange Multipliers to find the extreme values of $N_3(x)$ subject to the restrictions $N_1(x) = 0$, 
$N_2(x) = R^2.$
\end{proof}

Now let $f = \sigma_2 - 1, g = \sigma_3 - \sigma_1, 
I = (f), J = (g)$ and $\lambda= ()$.
We will show that
\begin{eqnarray*}
\lim_{n \rightarrow \infty} m_{0, \{\lambda\}_n}(V_n(I)) &=& 2, \\
\lim_{n \rightarrow \infty} m_{0, \{\lambda\}_n}(V_n(J)) &=& 3.
\end{eqnarray*}
Indeed, using Theorem \ref{thm:main}, it is enough to show that for $n \ge 3$, $b_0(\phi_n(f), \R^n) = 2$ and $b_0(\phi_n(g), \R^n) = 3$.

We take a fixed value of $x \in \R^n$ with $N_1(x) = \sigma_{1,n}(x) = 0$ and 
consider the polynomials
$$
f_x(t) = \sigma_{2,n}(x_1+t, \dots , x_n+t) - 1 = \binom{n}{2}t^2 -
\left(\frac12 N_2(x) + 1\right)
$$
and
$$
g_x(t) = \sigma_{3,n}(x_1+t, \dots , x_n+t) - \sigma_{1,n}(x_1+t, \dots , x_n+t) = $$
$$
= \binom{n}{3}t^3 -  \left(\frac{n-2}2 N_2(x) + n\right)t + 
\frac13 N_3(x).
$$
It is clear that $f_x$ has a positive discriminant, and on the other hand, ${\rm Disc}(g_x)$ is also positive since it is a positive multiple of 
$$
 4\left(\frac{n-2}2 N_2(x) + n\right)^3 - 3\binom{n}3N_3(x)^2 \ > \ 
 4\left(\frac{n-2}2 N_2(x)\right)^3 - 3\binom{n}3N_3(x)^2 \ \ge \ 0
$$
using Lemma \ref{lem:aux:fin}.

Finally, we split $\R^n$ as 
$$\R^n  = \bigcup_{x \in \R^n, \sigma_1(x) = 0} 
\{(x_1+t, \dots , x_n+t) \ | \ t\in \R \}$$ 
and then the claim follows using the continuity of roots with respect to the coefficients of a polynomial of fixed degree outside the region
where the discriminant vanishes. 

\section{Conclusion and open problems}
We have proved an upper bound of $2^{d-1}$  on the number of
semi-algebraically connected components of
a real hypersurface in $\R^n$  defined by a multi-affine polynomial 
of degree $d$. Moreover, we have shown that no bound which grows only polynomially with $n$ exists
for the 
higher Betti numbers of such hypersurfaces inside a closed ball.

Finally, we have proved a special case of 
a stability conjecture due to Basu and Riener on the 
cohomology modules of symmetric real algebraic sets. 

There are several open questions that are suggested by our results.

\begin{enumerate}[1.]
\item
Does the upper bound in Theorem~\ref{thm:ccez} extend to the bounded case ?
More precisely, is there a bound on $\beta_{\mathbf{A}_d,\mathbf{B},0}(n)$
which is independent of $n$ for some natural sequence $\mathbf{B}$, for example $\mathbf{B} = ([-1,1]^n)_{n > 0}$ ?
At the same time it would be interesting to extend  Theorem~\ref{thm:example} to the unbounded case.
More precisely, does there exist $c >1$, such that 
$\beta_{\mathbf{A}_d,p}(n) > c^n$ for some $d, p > 0$ ?
\item
Can one prove a bound
on the number of connected components of a real algebraic set in $\R^n$ defined by
\emph{two}  multi-affine polynomials of degree at most $d$ which is independent of
$n$ ?  We have shown that no such bound exists for real algebraic sets defined by three or more multi-affine polynomials. It would be satisfactory to be able fill this gap.

\item
Multi-affine polynomials that arise in practice (such as the basis
generating polynomial of a matroid) often have special properties 
such as real stability or being Lorentzian \cite{Anari-et-al,Branden-Huh}).
It would be interesting to study the topology of real hypersurfaces
defined by such polynomials from a quantitative point of view.

\item
The algorithmic problem of computing the number of semi-algebraically 
connected components of a given real algebraic set in $\R^n$ has attracted wide attention. The main tool for solving this problem is via computation of 
a one-dimensional semi-algebraic subset (called a roadmap of $V$). 
While there has been a steady improvement in the complexity of algorithms
for computing roadmaps of semi-algebraic sets \cite{BPR99,BRMS10,BMF2014},
the complexities of all known algorithms
are exponential in $n$. This is not unexpected as the number of semi-algebraically connected components of real algebraic sets in $\R^n$ defined by polynomials of degree at most $d$, grows exponentially in $n$ in the worst case for $d > 2$. However, in this paper we have proved that the number of semi-algebraically connected components of 
hypersurfaces defined by multi-affine polynomial is small. This suggests 
the problem of finding  a more efficient algorithm  (say with polynomial complexity) for
computing this number (maybe without resorting to a roadmap algorithm).
In the \emph{symmetric case}  such an algorithm (with polynomial complexity with the degree being considered fixed) was shown to exist in \cite{BC-focm}.

\end{enumerate}
\section*{Acknowledgements}
Basu was partially supported by NSF grants
CCF-1910441 and CCF-2128702.

\bibliographystyle{amsplain}
\bibliography{master}

\end{document}